\documentclass[12pt,english,a4paper]{smfart}

\usepackage[T1]{fontenc}
\usepackage{lmodern}
\usepackage{smfthm}
\usepackage[headings]{fullpage}
\usepackage{amssymb}

\newcommand{\Gal}{\operatorname{Gal}}
\newcommand{\Qp}{\mathbf{Q}_p}
\newcommand{\Zp}{\mathbf{Z}_p}
\newcommand{\Fp}{\mathbf{F}_p}
\newcommand{\ZZ}{\mathbf{Z}}
\newcommand{\FF}{\mathbf{F}}
\newcommand{\NN}{\mathbf{N}}
\newcommand{\RR}{\mathbf{R}}
\newcommand{\nbf}{\mathbf{n}}
\newcommand{\eps}{\varepsilon}
\newcommand{\OO}{\mathcal{O}}
\newcommand{\HH}{\mathcal{H}}
\newcommand{\Aut}{\operatorname{Aut}}
\newcommand{\Nm}{\operatorname{N}}
\newcommand{\sh}{\operatorname{sh}}
\newcommand{\orb}{\operatorname{orb}}
\newcommand{\cont}{\operatorname{cont}}

\newcommand{\alg}{\operatorname{alg}}
\newcommand{\Card}{\operatorname{Card}}
\newcommand{\Int}{\operatorname{Int}}
\newcommand{\LT}{\operatorname{LT}}
\newcommand{\val}{\operatorname{val}}
\newcommand{\hyphen}{\!\operatorname{-}\!}
\newcommand{\vp}{\val_p}
\newcommand{\vm}{\val_M}
\newcommand{\vx}{\val_X}
\newcommand{\vy}{\val_Y}
\newcommand{\ve}{\val_{\mathrm{E}}}
\newcommand{\bigO}{\operatorname{O}}
\newcommand{\dcroc}[1]{[\![ #1 ]\!]}
\newcommand{\dpar}[1]{(\!( #1 )\!)}
\newcommand{\efont}{\mathbf{E}}
\newcommand{\calE}{\mathcal{E}}
\newcommand{\kf}{E}

\newcommand{\et}{\widetilde{\mathbf{E}}}

\newcommand{\GL}{\operatorname{GL}}

\newcommand{\Gm}{\mathbf{G}_\mathrm{m}}
\renewcommand{\geq}{\geqslant}
\renewcommand{\leq}{\leqslant} 
\renewcommand{\phi}{\varphi} 
\renewcommand{\projlim}{\varprojlim} 

\newcommand{\cR}{\mathcal{R}}
\newcommand{\cT}{\mathcal{T}}

\newcommand{\LC}{\operatorname{LC}}

\NumberTheoremsIn{subsection}

\author{Laurent Berger}
\address{UMPA de l'ENS de Lyon \\
UMR 5669 du CNRS}
\email{laurent.berger@ens-lyon.fr}
\urladdr{perso.ens-lyon.fr/laurent.berger/}

\author{Sandra Rozensztajn}
\address{UMPA de l'ENS de Lyon \\
UMR 5669 du CNRS}
\email{sandra.rozensztajn@ens-lyon.fr}
\urladdr{perso.ens-lyon.fr/sandra.rozensztajn/}

\title[Super-H\"older vectors and the field of norms]{Super-H\"older vectors and \\ the field of norms}

\date{\today}

\begin{document}

\begin{abstract}
Let $\kf$ be a field of characteristic $p$. In a previous paper of ours, we defined and studied super-H\"older vectors in certain $\kf$-linear representations of $\Zp$. In the present paper, we define and study super-H\"older vectors in certain $\kf$-linear representations of a general $p$-adic Lie group. We then consider certain $p$-adic Lie extensions $K_\infty/K$ of a $p$-adic field $K$, and compute the super-H\"older vectors in the tilt of $K_\infty$. We show that these super-H\"older vectors are the perfection of the field of norms of $K_\infty/K$. By specializing to the case of a Lubin-Tate extension, we are able to recover $\kf \dpar{Y}$ inside the $Y$-adic completion of its perfection,  seen as a valued $E$-vector space endowed with the action of $\OO_K^\times$ given by the endomorphisms of the corresponding Lubin-Tate group.
\end{abstract}



\maketitle

\setcounter{tocdepth}{2}
\tableofcontents

\newpage

\setlength{\baselineskip}{18pt}
\section*{Introduction}

Let $\kf$ be a field of characteristic $p$, for example a finite field. In our paper \cite{BR22}, we defined and studied super-H\"older vectors in certain $\kf$-linear representations of the $p$-adic Lie group $\Zp$. These vectors are a characteristic $p$ analogue of locally analytic vectors. They allowed us to recover $\kf \dpar{X}$ inside the $X$-adic completion of its perfection, seen as a valued $E$-vector space endowed with the action of $\Zp^\times$ given by $a \cdot f(X) = f((1+X)^a-1)$.

In the present paper, we define and study super-H\"older vectors in certain $\kf$-linear representations of a general $p$-adic Lie group. We then consider certain $p$-adic Lie extensions $K_\infty/K$ of a $p$-adic field $K$, and compute the super-H\"older vectors in the tilt of $K_\infty$. We show that these super-H\"older vectors are the perfection of the field of norms of $K_\infty/K$. By specializing to the case of a Lubin-Tate extension, we are able to recover $\kf \dpar{Y}$ inside the $Y$-adic completion of its perfection,  seen as a valued $E$-vector space endowed with the action of $\OO_K^\times$ given by the endomorphisms of the corresponding Lubin-Tate group.

We now give more details about the contents of our paper. Let $\Gamma$ be a $p$-adic Lie group. It is known that $\Gamma$ always has a uniform open pro-$p$ subgroup $G$. Let $G$ be such a subgroup, and let $G_i = G^{p^i}$ for $i \geq 0$. Let $M$ be an $E$-vector space, endowed with a valuation $\vm$ such that $\vm(xm) = \vm(m)$ if $x \in \kf^\times$. We assume that $M$ is separated and complete for the $\vm$-adic topology. We say that a function $f : G \to M$ is super-H\"older if there exist constants $e > 0$ and $\lambda,\mu \in \RR$ such that $\vm(f(g)-f(h)) \geq p^\lambda \cdot p^{ei} +\mu$ whenever $gh^{-1} \in G_i$, for all $g,h \in G$ and $i \geq 0$. If $M$ is now endowed with an action of $G$ by isometries, and $m \in M$, we say that $m$ is a super-H\"older vector if the orbit map $g \mapsto g \cdot m$ is a super-H\"older function $G \to M$. We let $M^{G\hyphen e \hyphen \sh,\lambda}$ denote the space of super-H\"older vectors for given constants $e$ and $\lambda$ as in the definition above. The space of vectors of $M$ that are super-H\"older for a given $e$ is independent of the choice of the uniform subgroup $G$, and denoted by $M^{e \hyphen \sh}$. When $G=\Zp$ and $e=1$, we recover the definitions of \cite{BR22}. If $\Gamma$ is a $p$-adic Lie group and $e=1$, we get an analogue of locally $\Qp$-analytic vectors. If $K$ is a finite extension of $\Qp$, $\Gamma$ is the Galois group of a Lubin-Tate extension of $K$, and $e=[K:\Qp]$, we seem to get an analogue of locally $K$-analytic vectors.

From now on, assume that $p \neq 2$. Let $K$ be a $p$-adic field and let $K_\infty/K$ be an almost totally ramified $p$-adic Lie extension, with Galois group $\Gamma$ of dimension $d \geq 1$. The tilt of $K_\infty$ is the fraction field $\et_{K_\infty}$ of $\projlim_{x \mapsto x^p} \OO_{K_\infty} / p$. It is a perfect complete valued field of characteristic $p$, endowed with an action of $\Gamma$ by isometries. The field $\et_{K_\infty}$ naturally contains the field of norms $X_K(K_\infty)$ of the extension $K_\infty/K$, and it is known that $\et_{K_\infty}$ is the completion of the perfection of $X_K(K_\infty)$. We have the following result (theorem \ref{decfonsh}).

\begin{enonce*}{Theorem A} 
We have $\et_{K_\infty}^{d\hyphen\sh} = \cup_{n \geq 0} \phi^{-n}(X_K(K_\infty))$.
\end{enonce*}

Assume now that $K$ is a finite extension of $\Qp$, with residue field $k$, and let $\LT$ be a Lubin-Tate formal group attached to $K$. Let $K_\infty$ be the extension of $K$ generated by the torsion points of $\LT$, so that $\Gal(K_\infty/K)$ is isomorphic to $\OO_K^\times$. The field of norms $X_K(K_\infty)$ is isomorphic to $k \dpar{Y}$, and $\OO_K^\times$ acts on this field by the endomorphisms of the Lubin-Tate group: $a \cdot f(Y) = f([a](Y))$. Let $d=[K:\Qp]$. The following (theorem \ref{decprec}) is a more precise version of theorem A in this situation.

\begin{enonce*}{Theorem B} 
If $j \geq 1$, then $\et_{K_\infty}^{1+p^j \OO_K \hyphen d\hyphen\sh,dj} =k \dpar{Y}$.
\end{enonce*}

If $K=\Qp$ and $K_\infty/K$ is the cyclotomic extension, theorem B was proved in \cite{BR22}. A crucial ingredient of the proof of this theorem was Colmez' analogue of Tate traces for $\et_{K_\infty}$. If the Lubin-Tate group if of height $\geq 2$, there are no such traces (we state and prove a precise version of this assertion in \S \ref{sublt}). Instead of Tate traces, we a theorem of Ax and a precise characterization of the field of norms $X_K(K_\infty)$ inside $\et_{K_\infty}$ in order to prove theorem A.

As an application of theorem B, we compute the perfectoid commutant of $\Aut(\LT)$. If $b \in \OO_K^\times$ and $n \in \ZZ$, then $u(Y) = [b](Y^{q^n})$ is an element of $\et_{K_\infty}^+$ that satisfies the functional equation $u \circ [g](Y)  = [g] \circ u(Y)$ for all $g \in \OO_K^\times$. Conversely, we prove the following (theorem \ref{ltcom}).

\begin{enonce*}{Theorem C}
If $u \in \et_{K_\infty}^+$ is such that $\vy(u) > 0$ and $u \circ  [g] = [g] \circ u$ for all $g \in \OO_K^\times$, there exists $b \in \OO_K^\times$ and $n \in \ZZ$ such that $u(Y) = [b](Y^{q^n})$.
\end{enonce*}

In the last section, we give a characterization of super-H\"older functions on a uniform pro-$p$ group in terms of their Mahler expansions (theorem \ref{mahlthm}). In order to do so, we prove some results of independent interest on the space of continuous functions on $\OO_K^d$ with values in a valued $E$-vector space $M$ as above.

At the end of \cite{BR22}, we suggested an application of super-H\"older vectors for the action of $\Zp$ to the $p$-adic local Langlands correspondence for $\GL_2(\Qp)$. We hope that this general theory of super-H\"older vectors, especially in the Lubin-Tate case, will have applications to the $p$-adic local Langlands correspondence for other fields than $\Qp$.


\section{Super-H\"older functions and vectors}
\label{secshf}

In this section, we define Super-H\"older vectors inside a valued $E$-vector space $M$ endowed with an action of a $p$-adic Lie group $\Gamma$. The definition is very similar to the one that we gave for $\Gamma=\Zp$ in our paper \cite{BR22}. The main new technical tool is the existence of uniform open subgroups of $\Gamma$. These uniform subgroups look very much like $\Zp^d$ in a sense that we make precise.

\subsection{Uniform pro-$p$ groups}
\label{subspv}

Uniform pro-$p$ groups are defined at the beginning of \S 4 of \cite{P99}. We do not recall the definition, nor the notion of rank of a uniform pro-$p$ group, but rather point out the following properties of uniform pro-$p$ groups. A coordinate (below) is simply a homeomorphism.

\begin{prop}
\label{satpv}
If $G$ is a uniform pro-$p$ group of rank $d$, then
\begin{enumerate}
\item $G_i = \{ g^{p^i}$, $g \in G\}$ is an open normal (and uniform) subgroup of $G$ for $i \geq 0$
\item We have $[G_i:G_{i+1}] = p^d$ for $i \geq 0$
\item There is a coordinate $c : G \to \Zp^d$ such that $c(G_i) = (p^i \Zp)^d$ for  $i \geq 0$
\item If $g,h \in G$, then $gh^{-1} \in G_i$ if and only if $c(g)-c(h) \in (p^i \Zp)^d$
\end{enumerate}
\end{prop}

\begin{proof}
Properties (1-4) are proved in \S 4 of \cite{P99}. Alternatively, a uniform pro-$p$ group $G$ has a natural integer valued $p$-valuation $\omega$ such that $(G,\omega)$ is saturated (remark 2.1 of \cite{K05}). Properties (1-4) are then proved in \S 26 of \cite{S11}.
\end{proof}

For example, the pro-$p$ group $\Zp^d$ is uniform for all $d \geq 1$.

\begin{lemm}
\label{unifsub}
 If $G$ is a uniform pro-$p$ group, and $H$ is a uniform open subgroup of $G$, there exists $j \geq 0$ such that $G_{i+j} \subset H_i$ for all $i \geq 0$.
\end{lemm}

\begin{proof}
This follows from the fact that $\{G_i\}_{i \geq 0}$ forms a basis of neighborhoods of the identity in $G$. 
\end{proof}

A $p$-adic Lie group is a $p$-adic manifold that has a compatible group structure. For example, $\GL_n(\Zp)$ and its closed subgroups are $p$-adic Lie groups. We refer to \cite{S11} for a comprehensive treatment of the theory. Every uniform pro-$p$ group is a $p$-adic Lie group. Conversely, we have the following.

\begin{prop}
\label{exunif}
Every $p$-adic Lie group $\Gamma$ has a uniform open subgroup $G$, and the rank of $G$ is the dimension of $\Gamma$. 
\end{prop}

\begin{proof}
See Interlude A (pages 97--98) of \cite{P99}.
\end{proof}

\begin{prop}
\label{relunif}
Let $G$ be a pro-$p$ group of finite rank, and $N$ a closed normal subgroup of $G$. There exists an open subgroup $G'$ of $G$ such that $G'$, $G' \cap N$ and $G'/G' \cap N$ are all uniform.
\end{prop}

\begin{proof}
This is stated and proved on page 64 of \cite{P99} (their $H$ is our $G'$).
\end{proof}

\subsection{Super-H\"older functions and vectors}
\label{subsecshf}

Let $M$ be an $\kf$-vector space, endowed with a valuation $\vm$ such that $\vm(xm) = \vm(m)$ if $x \in \kf^\times$. We assume that $M$ is separated and complete for the $\vm$-adic topology. Throughout this {\S}, $G$ denotes a uniform pro-$p$ group.

\begin{defi}
\label{defsh}
We say that $f : G \to M$ is super-H\"older if there exist constants $\lambda,\mu \in \RR$ and $e > 0$ such that $\vm(f(g)-f(h)) \geq p^\lambda \cdot p^{ei} +\mu$ whenever $gh^{-1} \in G_i$, for all $g,h \in G$ and $i \geq 0$.
\end{defi}

\begin{rema}
\label{zpcase}
If $G=\Zp$ and $e=1$, we recover the functions defined in \S 1.1 \cite{BR22} (see also remark 1.12 of ibid). 

In the above definition, $e$ will usually be equal to either $1$ or $\dim(G)$.
\end{rema}

We let $\HH_e^{\lambda,\mu}(G,M)$ denote the space of functions such that $\vm(f(g)-f(h)) \geq p^\lambda \cdot p^{ei} +\mu$ whenever $gh^{-1} \in G_i$, for all $g,h \in G$ and $i \geq 0$, and $\HH_e^\lambda(G,M) = \cup_{\mu \in \RR} \HH_e^{\lambda,\mu}(G,M)$ and $\HH_e(G,M) = \cup_{\lambda \in \RR} \HH_e^\lambda(G,M)$. 

If $M,N$ are two valued $\kf$-vector spaces, and $f : M \to N$ is an $\kf$-linear map, we say that $f$ is H\"older-continuous if there exists $c>0$, $d \in \RR$ such that $\val_N(f(x)) \geq c \cdot \vm(x) +d$ for all $x \in M$.

\begin{prop}
\label{mton}
If $\pi : M \to N$ is a H\"older-continuous linear map, we get a map $\HH_e(G,M) \to \HH_e(G,N)$.
\end{prop}

\begin{proof}
Take $c,d \in \RR$ of H\"older continuity for $\pi$, $f \in \HH_e^{\lambda,\mu}(G,M)$,
and $g,h \in G$ with $gh^{-1} \in G_i$.
We have $\val_N(\pi(f(g))-\pi(f(h))) \geq c \cdot \val_M(f(g)-f(h)) + d \geq
cp^{\lambda}\cdot p^{ei}+(\mu+d)$, so that $\pi \circ f \in
\HH_e^{\lambda',\mu'}(G,N)$ with $p^{\lambda'} = cp^{\lambda}$, and
$\mu'=\mu+d$.
\end{proof}

\begin{prop}
\label{gtoh}
If $\alpha : G \to H$ is a group homomorphism, we get a map $\alpha^\ast : \HH_e(H,M) \to \HH_e(G,M)$.
\end{prop}

\begin{proof}
By definition of the subgroups $G_i$ and $H_i$, we have $\alpha(G_i)
\subset H_i$ for all $i$. Take $f \in \HH_e^{\lambda,\mu}(H,M)$, and $g,h \in G$ with
$gh^{-1} \in G_i$. We have $\val_M(f(\alpha(g))-f(\alpha(h))) \geq
p^\lambda\cdot p^{ei}+\mu$ as $\alpha(g)\alpha(h)^{-1} \in H_i$, so that $\alpha^\ast(f) = f
\circ \alpha \in \HH_e^{\lambda,\mu}(G,M)$.
\end{proof}

\begin{prop}
\label{hrng}
Suppose that $M$ is a ring, and that $\vm(mm') \geq \vm(m) + \vm(m')$ for all $m,m' \in M$. If $c \in \RR$, let $M_c= M^{\vm \geq c}$. 

\begin{enumerate}
\item If $f \in \HH_e^{\lambda,\mu}(G,M_c)$ and $g \in \HH_e^{\lambda,\nu}(G,M_d)$, and $\xi=\min(\mu+d,\nu+c)$, then $fg \in \HH_e^{\lambda,\xi}(G,M_{c+d})$.
\item If $\lambda,\mu \in \RR$, then $\HH_e^{\lambda,\mu}(G,M_0)$ is a subring of $C^0(G,M)$.
\item If $\lambda \in \RR$, then $\HH_e^{\lambda}(G,M)$ is a subring of $C^0(G,M)$.
\end{enumerate}
\end{prop}

\begin{proof}
Items (2) and (3) follow from item (1), which we now prove.
If $x,y \in G$, then
\[ (fg)(x) - (fg)(y) = (f(x)-f(y))g(x) + (g(x)-g(y))f(y),\]
which implies the claim.
\end{proof}

We now assume that $M$ is endowed with an $\kf$-linear action by isometries of $G$. If $m \in M$, let $\orb_m : G \to M$ denote the function defined by $\orb_m(g) = g \cdot m$. 

\begin{defi}
\label{defshvec}
Let $M^{G\hyphen e \hyphen\sh,\lambda,\mu}$ be those $m \in M$ such that $\orb_m \in \HH_e^{\lambda,\mu}(G,M)$, and let $M^{G\hyphen e \hyphen\sh,\lambda}$ and $M^{G\hyphen e \hyphen\sh}$ be the corresponding sub-$\kf$-vector spaces of $M$.
\end{defi}

\begin{rema}
\label{mcont}
We assume that $G$ acts by isometries on $M$, but not that $G$ acts continuously on $M$, namely that $G \times M \to M$ is continuous. However, let $M^{\cont}$ denote the set of $m \in M$ such that $\orb_m : G \to M$ is continuous. It is easy to see that $M^{\cont}$ is a closed sub-$\kf$-vector space of $M$, and that $G \times M^{\cont} \to M^{\cont}$ is continuous (compare with \S 3 of \cite{E17}). We then have $M^{\sh} \subset M^{\cont}$.
\end{rema}

\begin{lemm}
\label{mshalt}
If $m \in M$, then $m \in M^{G\hyphen e \hyphen\sh,\lambda,\mu}$ if and only if for all $i \geq 0$, we have $\vm(g \cdot m - m) \geq p^\lambda \cdot p^{ei} +\mu$ for all $g \in G_i$.
\end{lemm}

\begin{proof}
If $m \in M$, then $m \in M^{G\hyphen e \hyphen\sh,\lambda,\mu}$ if and
only if the function $\orb_m$ is in $\HH_e^{\lambda,\mu}(G,M)$, that is,
for all $g,h$ with $gh^{-1} \in G_i$, we have 
$\vm(g\cdot m - h \cdot m) \geq p^{\lambda}\cdot p^{ei}+\mu$. 
As $G$ acts by isometries, we have
$\vm(g\cdot m - h \cdot m) = \vm(h^{-1}g \cdot m - m)$. The result
follows, as $h^{-1}g = h^{-1} \cdot gh^{-1} \cdot h \in G_i$.
\end{proof}

\begin{lemm}
\label{mshclo}
The space $M^{G\hyphen e \hyphen\sh,\lambda,\mu}$ is a closed sub-$E$-vector space of $M$.
\end{lemm}

\begin{lemm}
\label{lamsub}
If $i_0 \geq 0$, and $m \in M$ is such that $\vm(g \cdot m - m) \geq p^\lambda \cdot p^{ei} +\mu$ for all $g \in G_i$ with $i \geq i_0$, then $m \in M^{G\hyphen e \hyphen\sh,\lambda}$.
\end{lemm}

\begin{proof}
Take $i < i_0$, and let $R_i$ be a set of
representatives of $G_{i_0}\backslash G_i$. This is a finite set, so there
exists $\mu_i \in \RR$ such that $\vm(r \cdot m -m) \geq
p^{\lambda}\cdot p^{ei}+\mu_i$ for all $r \in R_i$. If $g \in G_i$,
it can be written as $g = hr$ for some $h\in G_{i_0}$ and $r\in
R_i$. We then have $g \cdot m - m = hr \cdot m - h \cdot m + h\cdot m - m$,
so that $\vm (g \cdot m - m) \geq 
\min (\vm (r \cdot m - m), \vm (h \cdot m -m))$ (recall that $G$
acts by isometries),
so $\vm (g \cdot m - m) \geq 
\min(p^\lambda \cdot p^{ei} + \mu_i,p^\lambda\cdot p^{ei_0}+\mu) \geq
p^\lambda\cdot p^{ei}+ \min(\mu,\mu_i)$ as $i_0 > i$. If $\mu'$ is the min of $\mu$ and the $\mu_i$ for $0 \leq i < i_0$,
then $m \in M^{G\hyphen e \hyphen\sh,\lambda,\mu'}$.
\end{proof}

Recall that if $k \geq 0$, then $G_k$ is also a uniform pro-$p$ group.

\begin{lemm}
\label{shsbgr}
If $k \geq 0$ then $M^{G\hyphen e \hyphen\sh,\lambda} = M^{G_k\hyphen e \hyphen \sh,\lambda+k}$.
\end{lemm}

\begin{proof}
Note that $(G_k)_i =  G_{i+k}$. The inclusion $M^{G\hyphen e \hyphen\sh,\lambda} \subset M^{G_k\hyphen e \hyphen \sh,\lambda+k}$ is obvious, and the reverse inclusion follows from lemma \ref{lamsub}.
\end{proof}

\begin{prop}
\label{indepleas}
The space $M^{H\hyphen e \hyphen\sh}$ does not depend on the choice of a uniform open subgroup $H \subset G$. 
\end{prop}

\begin{proof}
Let $H$ and $H'$ be uniform open subgroups of $G$. The group $H \cap H'$ contains an open uniform subgroup by prop \ref{exunif}, so to prove the proposition, we can further assume that $H' \subset H$. We then have $H'_i \subset H_i$ for all $i$, so that if $m \in M^{H\hyphen e\hyphen \sh,\lambda,\mu}$, then $m \in M^{H'\hyphen e\hyphen \sh,\lambda,\mu}$.
This implies that $M^{H\hyphen e\hyphen \sh, \lambda} \subset M^{H'\hyphen e\hyphen \sh, \lambda}$. Conversely, by lemma \ref{unifsub}, there exists $j$ such that $H_j \subset H'$. The previous reasoning implies that $M^{H'\hyphen e\hyphen \sh, \lambda} \subset M^{H_j\hyphen e\hyphen \sh, \lambda}$. Lemma \ref{shsbgr} now implies that $M^{H_j\hyphen e\hyphen \sh, \lambda} =
M^{H\hyphen e\hyphen \sh, \lambda-j}$. 

These inclusions imply the proposition.
\end{proof}

\begin{defi}
\label{shplg}
If $\Gamma$ is a $p$-adic Lie group that acts by isometries on $M$, we let $M^{e\hyphen\sh} = M^{G \hyphen e \hyphen\sh}$ where $G$ is any uniform open subgroup of $\Gamma$.
\end{defi}

\begin{rema}
\label{efmsh}
If $e \leq f$, then $M^{f\hyphen\sh} \subset M^{e\hyphen\sh}$.
\end{rema}

Recall that $G$ is a uniform pro-$p$ group. If a closed normal subgroup $N$ of $G$ acts trivially on $M$, then $G/N$ acts on $M$.

\begin{prop}
\label{quotsh}
If a closed normal subgroup $N$ of $G$ acts trivially on $M$, then $M^{G\hyphen e \hyphen\sh} = M^{G/N \hyphen e \hyphen\sh}$.
\end{prop}

\begin{proof}
By prop \ref{relunif}, $G$ has an open subgroup $G'$ such that $G'$ and $G'/N'$ are uniform (where $N'=G' \cap N$). By prop \ref{indepleas}, we have $M^{G\hyphen e \hyphen\sh} = M^{G'\hyphen e \hyphen\sh}$ and $M^{G/N \hyphen e \hyphen\sh} = M^{G'/N' \hyphen e \hyphen\sh}$. Let $\pi : G' \to G'/N'$ denote the projection. We have $\pi(G'_i) = (G'/N')_i$ for 	all $i$. Hence if $m \in M$, then $\vm(g\cdot m-m) \geq p^\lambda\cdot p^{ei}+ \mu$ for all $g \in G'_i$ if and only if $\vm(\pi(g) \cdot m-m) \geq p^\lambda\cdot p^{ei}+ \mu$ for all $\pi(g) \in (G'/N')_i$.
\end{proof}

\begin{prop}
\label{propmsh}
Suppose that $M$ is a ring, and that $g(mm') = g(m)g(m')$ and $\vm(mm') \geq \vm(m) + \vm(m')$ for all $m,m' \in M$ and $g \in G$.
\begin{enumerate}
\item If $v \in \RR$ and $m,m' \in M^{G\hyphen e \hyphen\sh,\lambda,\mu} \cap M^{\vm \geq v}$, then $m \cdot m' \in M^{G\hyphen e \hyphen\sh,\lambda,\mu+v}$.
\item If $m \in M^{G\hyphen e \hyphen\sh,\lambda,\mu} \cap M^\times$, then $1/m \in M^{G\hyphen e \hyphen\sh,\lambda,\mu-2 \vm(m)}$.
\end{enumerate}
\end{prop}

\begin{proof}
Item (1) follows from prop \ref{hrng}
and lemma \ref{mshalt}. Item (2) follows from 
\[ g\left(\frac 1m\right) - \frac 1m = \frac{m-g(m)}{g(m) m}. \]
\end{proof}

\section{The field of norms}
\label{secfon}

Let $K$ be a $p$-adic field, and let $K_\infty$ be an algebraic Galois extension of $K$, whose Galois group $G$ is a $p$-adic Lie group of dimension $\geq 1$. We assume that $K_\infty/K$ is almost totally ramified, namely that the inertia subgroup of $G$ is open in $G$. Let $d=\dim(G)$ and let $\ell=p^d$. Let $\et^+_{K_\infty}$ denote the ring $\projlim_{x \mapsto x^\ell} \OO_{K_\infty} / p$. This is a perfect domain of characteristic $p$, which has a natural action of $G$. The map $(y_j)_{j \geq 0} \mapsto (y_{di})_{i \geq 0}$ gives an isomorphism between $\projlim_{x \mapsto x^p} \OO_{K_\infty} / p$ and $\et^+_{K_\infty}$, so that $\et^+_{K_\infty}$ is the ring of integers of the tilt of $\hat{K}_\infty$ (see \S 3 of \cite{S12}).

If $x=(x_i)_{i \geq 0}$, and $\hat{x}_i$ is a lift of $x_i$ to $\OO_{K_\infty}$, then $\ell^i \vp(\hat{x}_i)$ is independent of $i \geq 0$ such that $x_i \neq 0$.
We define a valuation on $\et^+_{K_\infty}$ by $\ve(x) = \lim_{i \to + \infty} \ell^i \vp(\hat{x}_i)$.

The aim of this section is to compute $(\et^+_{K_\infty})^{d \hyphen\sh}$. Given definition \ref{shplg}, we assume from now on (replacing $K$ by a finite subextension if necessary) that $G$ is uniform and that $K_\infty/K$ is totally ramified. Let $k$ denote the common residue field of $K$ and $K_\infty$.

\subsection{The field of norms}
\label{subdefon}

Let $\calE(K_\infty)$ denote the set of finite extensions $E$ of $K$ such that $E \subset K_\infty$. Let $X_K(K_\infty)$ denote the set of sequences $(x_E)_{E \in \calE(K_\infty)}$ such that $x_E \in E$ for all $E \in \calE(K_\infty)$, and $\Nm_{F/E}(x_F) = x_E$ whenever $E \subset F$ with $E,F \in \calE(K_\infty)$. 

If $n \geq 0$, let $K_n  = K_\infty^{G_n}$ so that $[K_{n+1}:K_n]=\ell$, $\{K_n\}_{n \geq 0}$ is a cofinal subset of $\calE(K_\infty)$, and $X_K(K_\infty) = \projlim_{\Nm_{K_n/K_{n-1}}} K_n$. If $x= (x_n)_{n \geq 0} \in X_K(K_\infty)$, let $\ve(x)=\vp(x_0)$. 

\begin{theo}
\label{wintheo}
Let $K$ and $K_\infty$ be as above.
\begin{enumerate}
\item If $x,y \in X_K(K_\infty)$, then $\{\Nm_{K_{n+j}/K_n}(x_{n+j}+y_{n+j})\}_{j \geq 0}$ converges for all $n \geq 0$. 
\item If we set $(x+y)_n = \lim_{j \to +\infty} \Nm_{K_{n+j}/K_n}(x_{n+j}+y_{n+j})$, then $x+y \in X_K(K_\infty)$, and the set $X_K(K_\infty)$ with this addition law, and componentwise multiplication, is a field of characteristic $p$.
\item The function $\ve$ is a valuation on $X_K(K_\infty)$, for which it is complete
\item If $\varpi = (\varpi_n)_{n \geq 0}$ is a norm compatible sequence of uniformizers of $\OO_{K_n}$, the valued field $X_K(K_\infty)$ is isomorphic to $k \dpar{\varpi}$ (with $\val(\varpi)=\vp(\varpi_0)$).
\end{enumerate}
\end{theo}

\begin{proof}
By a result of Sen \cite{S72}, $K_\infty/K$ is strictly APF in the terminology of \S 1.2 of \cite{W83} (see 1.2.2 of ibid). The theorem is then proved in \S 2 of ibid.
\end{proof}

Let $X_K^+(K_\infty) = \projlim_{\Nm_{K_n/K_{n-1}}} \OO_{K_n}$ be the ring of integers of the valued field $X_K(K_\infty)$.

If $c>0$, let $I^c_n = \{ x \in \OO_{K_n}$ such that $\vp(x) \geq c\}$. If $m, n \geq 0$, the map $\OO_{K_n} / I^c_n \to \OO_{K_{m+n}} / I^c_{m+n}$ is well-defined and injective.

\begin{prop}
\label{fonembone}
There exists $c(K_\infty/K) \leq 1$ such that if $0 < c \leq c(K_\infty/K)$, then $\vp(\Nm_{K_{n+k}/K_n}(x)/x^{[K_{n+k}:K_n]}-1) \geq c$ for all $n,k \geq 0$ and $x \in \OO_{K_{n+k}}$.
\end{prop}

\begin{proof}
See \cite{W83} as well as \S 4 of \cite{CD15}. The result follows from the fact (see 1.2.2 of \cite{W83}) that the extension $K_\infty/K$ is strictly APF. One can then apply 1.2.1, 4.2.2 and 1.2.3 of \cite{W83}.
\end{proof}

Using prop \ref{fonembone}, we get a map $\iota : X_K^+(K_\infty) \to \projlim_{x \mapsto x^\ell} \OO_{K_\infty} / I^c_\infty$ given by $(x_n)_{n \geq 0} \in \projlim_{\Nm_{K_n/K_{n-1}}} \OO_{K_n} \mapsto (\overline{x}_n)_{n \geq 0}$. Let $\projlim_{x \mapsto x^\ell} \OO_{K_n} / I^c_n$ denote the set of $(x_n)_{n \geq 0} \in \projlim_{x \mapsto x^\ell} \OO_{K_\infty} / I^c_\infty$ such that $x_n \in \OO_{K_n} / I^c_n$ for all $n \geq 0$.

\begin{prop}
\label{fonembtwo}
Let $0 < c \leq c(K_\infty/K)$ be as in prop \ref{fonembone}.
\begin{enumerate}
\item the natural map $\et^+_{K_\infty} \to \projlim_{x \mapsto x^\ell} \OO_{K_\infty} / I^c_\infty$ is a bijection
\item the map $\iota : X_K^+(K_\infty) \to \projlim_{x \mapsto x^\ell} \OO_{K_\infty} / I^c_\infty = \et^+_{K_\infty}$ is injective and isometric
\item the image of $\iota$ is $\projlim_{x \mapsto x^\ell} \OO_{K_n} / I^c_n$.
\end{enumerate}
\end{prop}

\begin{proof}
See \cite{W83} and \S 4 of \cite{CD15}. We give a few more details for the convenience of the reader. Item (1) is classical (see for instance prop 4.2 of \cite{CD15}).The map $\iota$ is obviously injective and isometric. For (3), choose $x=(x_n)_{n \geq 0} \in \projlim_{x \mapsto x^\ell} \OO_{K_n} / I^c_n$, and choose a lift $\hat{x}_n \in \OO_{K_n}$ of $x_n$. One proves that $\{\Nm_{K_{n+j}/K_n}(\hat{x}_{n+j})\}_{j \geq 0}$ converges to some $y_n \in \OO_{K_n}$, and that $(y_n)_{n \geq 0} \in X_K^+(K_\infty)$ is a lift of $(x_n)_{n \geq 0}$. See \S 4 of \cite{CD15} for details, for instance the proof of lemma 4.1.
\end{proof}

Prop \ref{fonembtwo} allows us to see $X_K^+(K_\infty)$, and hence $\phi^{-n}(X_K^+(K_\infty))$ for all $n \geq 0$, as a subring of $\et^+_{K_\infty}$.

\begin{prop}
\label{perfcomp}
The ring $\cup_{n \geq 0} \phi^{-n}(X_K^+(K_\infty))$ is dense in $\et^+_{K_\infty}$.
\end{prop}

\begin{proof}
See \S 4.3 of \cite{W83}.
\end{proof}

\subsection{Decompleting the tilt}
\label{subdecfon}

We now compute $(\et^+_{K_\infty})^{d \hyphen \sh}$. Since prop \ref{axlb} below is vacuous if $p=2$, we assume in this {\S} that $p \neq 2$.

\begin{prop}
\label{axlb}
If $0 < c \leq 1-1/(p-1)$, and $x \in \OO_{K_\infty}$ is such that $\vp(g(x)-x) \geq 1$ for all $g \in G_n$, then the image of $x$ in $\OO_{K_\infty} / I^c_\infty$ belongs to $\OO_{K_n} / I^c_n$.
\end{prop}

\begin{proof}
If $\vp(g(x)-x) \geq 1$ for all $g \in \Gal(K^{\alg}/K_n)$, then by theorem 1.7 of \cite{LB10} (an optimal version of a theorem of Ax), there exists $y \in K_n$ such that $\vp(x-y) \geq 1-1/(p-1)$. This implies the proposition.
\end{proof}

\begin{prop}
\label{xksh}
If $c = p^\gamma$ is as above, then $X_K^+(K_\infty) \subset (\et^+_{K_\infty})^{G\hyphen d \hyphen\sh,\gamma,0}$.
\end{prop}

\begin{proof}
Take $x= (x_n)_{n \geq 0} \in \projlim_{x \mapsto x^\ell} \OO_{K_n} / I^c_n$. If $g \in G_i$, then $g(x_n) = x_n$ for $n \leq i$, so that $\ve(gx-x) \geq p^{di}p^\gamma$.
\end{proof}

\begin{theo}
\label{decfonsh}
We have 
\begin{enumerate}
\item $(\et^+_{K_\infty})^{G\hyphen d \hyphen\sh,0,0} \subset X_K^+(K_\infty)$
\item $(\et^+_{K_\infty})^{d \hyphen \sh} = \cup_{n \geq 0} \phi^{-n}(X_K^+(K_\infty))$ and $\et_{K_\infty}^{d \hyphen \sh} = \cup_{n \geq 0} \phi^{-n}(X_K(K_\infty))$
\end{enumerate}
\end{theo}

\begin{proof}
Take $c \leq \min(c(K_\infty/K),1-1/(p-1))$. Take $x=(x_n)_{n \geq 0} \in \projlim_{x \mapsto x^\ell} \OO_{K_\infty} / p$. If $n \geq 0$ and $x \in  (\et^+_{K_\infty})^{G\hyphen d \hyphen\sh,0,0}$, then $\ve(g(x)-x) \geq p^{dn}$ if $g \in G_n$. This implies that $\vp(g(x_n)-x_n) \geq 1$ if $g \in G_n$. By prop \ref{axlb}, the image of $x_n$ in $\OO_{K_\infty} / I^c_\infty$ belongs to $\OO_{K_n} / I^c_n$. Hence the image of $x$ in $\projlim_{x \mapsto x^\ell} \OO_{K_\infty} / I^c_\infty$ belongs to $\projlim_{x \mapsto x^\ell} \OO_{K_n} / I^c_n$. By prop \ref{fonembtwo}, $x$ belongs to $X_K^+(K_\infty)$. This proves (1).

Since $\ve(\phi(x)) = p \cdot \ve(x)$, item (2) follows from (1) and props \ref{xksh} and \ref{propmsh}.
\end{proof}

\begin{rema}
\label{otherlocan}
We have $\et_{K_\infty}^{d\hyphen\sh} \subset \et_{K_\infty}^{1\hyphen\sh}$. The field $\et_{K_\infty}^{1\hyphen\sh}$ contains the field of norms $X_K(L_\infty)$ of any $p$-adic Lie extension $L_\infty/K$ contained in $K_\infty$. Indeed, $\et_{L_\infty} \subset \et_{K_\infty}$ and if $e=\dim \Gal(L_\infty/K)$, then $X_K(L_\infty) \subset \et_{L_\infty}^{e\hyphen\sh} \subset \et_{K_\infty}^{1\hyphen\sh}$ (see prop \ref{quotsh}).

Can one give a description of $\et_{K_\infty}^{1\hyphen\sh}$, for example along the lines of \S 5 of \cite{LB16}?
\end{rema}

\section{The Lubin-Tate case}
\label{seclt}

We now specialize the constructions of the previous section to the case when $K_\infty$ is generated over $K$ by the torsion points of a Lubin-Tate formal group.

\subsection{Lubin-Tate formal groups}
\label{subltfg}

Let $K$ be a finite extension of $\Qp$ of degree $d$, with ring of integers $\OO_K$, inertia index $f$, ramification index $e$, and residue field $k$. Let $q=p^f=\Card(k)$ and let $\pi$ be a uniformizer of $\OO_K$. Let $\LT$ be the Lubin-Tate formal $\OO_K$-module attached to $\pi$ (see \cite{LT65}). We choose a coordinate $Y$ on $\LT$. For each $a \in \OO_K$ we get a power series $[a](Y) \in \OO_K \dcroc{Y}$, that we now see as an element of $k \dcroc{Y}$. In particular, $[\pi](Y)=Y^q$. Let $S(T,U) \in k \dcroc{T,U}$ denote the reduction mod $\pi$ of the power series giving the addition law in $\LT$ in that coordinate. Recall that $S(T,0) = T$ and $S(0,U) = U$.

\begin{lemm}
\label{diflt}
If $a,b \in \OO_K$ and $i \geq 0$, then $\vy([a+p^i b](Y)-[a](Y)) \geq p^{di}$.

Furthermore, $[1+\pi^i](Y)=Y+Y^{q^i}+\bigO(Y^{q^i+1})$.
\end{lemm}

\begin{proof}
We have $[\pi](Y) = Y^q$, so $\vy([\pi](Y)) \geq p^f$. Writing $p = u\pi^e$
for a unit $u$, we see that $\vy([p^i b](Y)) \geq p^{di}$ if $b \in \OO_K$.
If $a,b \in \OO_K$ and $i \geq 0$, then $[a+bp^i](Y) = S([a](Y),[bp^i](Y))$. We have $S(T,U) = T+U+TU \cdot R(T,U)$, so that  $[a+bp^i](Y) - [a](Y) = S([a](Y),[bp^i](Y))-[a](Y) \in [bp^i](Y) \cdot  k \dcroc{Y}$.
This implies the first result. 

The second claim follows likewise from the fact that $[1+\pi^i](Y) = S(Y,[\pi^i](Y)) = Y + [\pi^i](Y) + Y \cdot [\pi^i](Y) \cdot R(Y,[\pi^i](Y))$.
\end{proof}

Let $\efont = k \dpar{Y}$. Let $\efont_n = k \dpar{Y^{1/q^n}}$ and let $\efont_\infty = \cup_{n \geq 0} \efont_n$. These fields are endowed with the $Y$-adic valuation $\vy$, and we let $\efont_\star^+$ denote the ring of integers of $\efont_\star$. The group $\OO_K^\times$ acts on $\efont_n$ by $a \cdot f(Y^{1/q^n}) = f([a](Y^{1/q^n}))$. 

\begin{lemm}
\label{sbokc}
If $j \geq 1$ ($j \geq 2$ if $p=2$), then $1+p^j \OO_K$ is uniform, and $(1+p^j \OO_K)_i = 1+p^{i+j} \OO_K$.
\end{lemm}

\begin{proof}
The map $1+p^j\OO_K \to \OO_K$, given by $x \mapsto p^{-j} \cdot \log_p(x-1)$, is an
isomorphism of pro-$p$ groups taking $1+p^{i+j} \OO_K$ to $p^i\OO_K$.
\end{proof}

Recall that $d=[K:\Qp]$, that $f=[k:\Fp]$, and that $q=p^f$.

\begin{prop}
\label{enpsh}
We have $\efont_n^+ = (\efont_n^+)^{1+p^j \OO_K \hyphen d \hyphen \sh, dj-fn,0}$.
\end{prop}

\begin{proof}
If $b \in \OO_K$ and $i,j \geq 0$, then by lemma \ref{diflt}, we have 
\[ \vy([1+p^{i+j} b](Y^{1/q^n})-Y^{1/q^n}) \geq 1/q^n \cdot p^{d(i+j)} = p^{dj-fn} \cdot p^{di}. \]
Lemma \ref{sbokc} then implies that $Y^{1/q^n} \in (\efont_n^+)^{1+p^j \OO_K \hyphen d \hyphen \sh, dj-fn,0}$. The lemma now follows from prop \ref{propmsh} and lemma \ref{mshclo}.
\end{proof}

\begin{coro}
\label{egsk}
We have $\efont = \efont^{1+p^j \OO_K \hyphen d \hyphen \sh,dj}$.
\end{coro}

\begin{proof}
This follows from prop \ref{enpsh} with $n=0$, and prop \ref{propmsh}.
\end{proof}

\begin{prop}
\label{levelzer}
If $\eps>0$, then $k \dcroc{Y}^{1+p^j \OO_K \hyphen d \hyphen \sh, dj + \eps} \subset k \dcroc{Y^p}$. 
\end{prop}

\begin{proof}
Take $f(Y) \in k \dcroc{Y}$. There is a power series $h(T,U) \in k \dcroc{T,U}$ such that 
\[ f(T+U) = f(T) + U \cdot f'(T) + U^2 \cdot h(T,U). \]
If $m \geq 0$, lemma \ref{diflt} implies that $[1+\pi^m](Y) = Y + Y^{q^m} + \bigO(Y^{q^m+1})$. Therefore,
\[ f([1+\pi^m](Y)) = f(Y) + (Y^{q^m} + \bigO(Y^{q^m+1})) \cdot f'(Y) + \bigO(Y^{2 q^m}). \]
If $f(Y) \notin k \dcroc{Y^p}$, then $f'(Y) \neq 0$. Let $\mu=\vy(f'(Y))$. The above computations imply that $\vy(f([1+\pi^{ei+ej}](Y)) - f(Y)) = p^{dj} \cdot p^{di}+\mu$ for $i \gg 0$. 

This implies the claim, since $\pi^e \OO_K = p \OO_K$.
\end{proof}

\begin{coro}
\label{etshlevel}
We have $\efont_\infty^{1+p^j \OO_K \hyphen d \hyphen \sh, dj-fn} = \efont_n$.
\end{coro}

\begin{proof}
We prove that, more generally, $\efont_\infty^{1+p^j \OO_K \hyphen d \hyphen \sh, dj-\ell} = k \dpar{Y^{1/p^\ell}}$. 
Take $f(Y^{1/p^m}) \in (\efont^+_\infty)^{1+p^j \OO_K \hyphen d \hyphen \sh, dj-\ell}$
where $f(Y) \in k \dcroc{Y}$. Since $\vy(h^p) = p \cdot \vy(h)$ for all
$h \in \et^+$, we have $f^{p^m}(Y) \in (\efont^+_\infty)^{1+p^j \OO_K \hyphen d \hyphen \sh, dj-\ell+m}$, 
where $f^{p^m}(Y) \in \kf \dcroc{Y}$ is 
$f^{p^m}(Y) = f(Y^{1/p^m})^{p^m}$. If $m \geq \ell+1$, then prop
\ref{levelzer} implies that $f^{p^m}(Y) \in \kf \dcroc{Y^p}$, so that
$f(Y) = g(Y^p)$, and $f(Y^{1/p^m}) = g(Y^{1/p^{m-1}})$. 
This implies the
claim.
\end{proof}

\subsection{Decompletion of $\et$}
\label{sublt}

Since we use the results of \S\ref{subdecfon}, we once more assume that $p \neq 2$.
Let $\et$ denote the $Y$-adic completion of $\efont_\infty$. 

\begin{theo}
\label{decprec}
We have $\et^{1+p^j \OO_K \hyphen d \hyphen \sh,dj} = \efont$, and $\et^{d\hyphen\sh} = \efont_\infty$.
\end{theo}

\begin{proof}
Let $K_\infty = K(\LT[\pi^\infty])$ denote the extension of $K$ generated by the torsion points of $\LT$, and let $\Gamma = \Gal(K_\infty/K)$. The Lubin-Tate character $\chi_\pi$ gives rise to an isomorphism $\chi_\pi : \Gamma \to \OO_K^\times$. For $n \geq 1$, let $K_n = K(\LT[\pi^n])$. If $(\pi_n)_{n \geq 1}$ is a compatible sequence of primitive $\pi^n$-torsion points of $\LT$, then $\pi_n$ is a uniformizer of $\OO_{K_n}$, $\varpi=(\pi_n)_{n \geq 0}$ belongs to $\projlim_{\Nm_{K_n/K_{n-1}}} \OO_{K_n}$, and $X_K(K_\infty) = k \dpar{\varpi}$ by theorem \ref{wintheo}. If $g \in \Gamma$, then $g(\varpi) = [\chi_\pi(g)](\varpi)$, so that if we identify $\Gamma$ and $\OO_K^\times$, then $X_K(K_\infty) = \efont$ with its action of $\OO_K^\times$. Prop \ref{perfcomp} implies that $\et = \et_{K_\infty}$ as valued fields with an action of (an open subgroup of) $\OO_K^\times$. We can therefore apply theorem \ref{decfonsh}, and get $(\et^+)^{d\hyphen\sh} = \efont^+_\infty$. This implies the second statement. The first one then follows from coro \ref{etshlevel}.
\end{proof}

\begin{rema}
\label{cofin}
In the above proof, note that $K_\infty^{1+p^n \OO_K} = K_{ne}$, so that the numbering is not the same as in \S\ref{subdefon}.
\end{rema}

\begin{rema}
\label{gamodlt}
We can define Lubin-Tate $\Gamma$-modules over $\efont$ as in \S 3.2 of \cite{BR22}. The results proved in that section carry over to the Lubin-Tate setting without difficulty.
\end{rema}

In theorem 2.9 of \cite{BR22}, we proved theorem \ref{decprec} above in the cyclotomic case, using Tate traces. There are no such Tate traces in the Lubin-Tate case if $K \neq \Qp$. We now explain why this is so. More precisely, we prove that there is no $\Gamma$-equivariant $k$-linear projector $\et \to \efont$ if $K \neq \Qp$. Choose a coordinate $T$ on $\LT$ such that $\log_{\LT}(T) = \sum_{n \geq 0} T^{q^n}/\pi^n$, so that $\log'_{\LT}(T) \equiv 1 \bmod{\pi}$. Let $\partial = 1/\log'_{\LT}(T) \cdot d/dT$ be the invariant derivative on $\LT$. Let $\phi_q=\phi^f$ where $q=p^f$.

\begin{lemm}
\label{ltder}
We have $d \gamma(Y)/dY \equiv \chi_\pi(\gamma)$ in $\efont$ for all $\gamma \in \Gamma$.
\end{lemm}

\begin{proof}
Since $\log'_{\LT} \equiv 1 \bmod{\pi}$, we have $\partial = d/dY$ in $\efont$. Applying $\partial \circ \gamma =  \chi_\pi(\gamma) \gamma \circ \partial$ to $Y$, we get the claim.
\end{proof}

\begin{lemm}
\label{pginv}
If $\gamma \in \Gamma$ is nontorsion, then $\efont^{\gamma=1} = k$.
\end{lemm}

\begin{prop}
\label{noltrace}
If $K \neq \Qp$, there is no $\Gamma$-equivariant map $R : \efont \to \efont$ such that $R(\phi_q(f)) = f$ for all $f \in \efont$.
\end{prop}

\begin{proof}
Suppose that such a map exists, and take $\gamma \in \Gamma$ nontorsion and such that $\chi_\pi(\gamma) \equiv 1 \bmod{\pi}$. We first show that if $f \in \efont$ is such that $(1- \gamma) f \in \phi_q(\efont)$, then $f \in \phi_q(\efont)$. Write $f=f_0 + \phi_q( R(f))$ where $f_0 = f- \phi_q(R(f))$, so that $R(f_0)=0$ and $(1- \gamma) f_0  = \phi_q(g) \in \phi_q(\efont)$. Applying $R$, we get $0 = (1- \gamma) R(f_0) = g$. Hence $g=0$ so that $(1- \gamma) f_0 = 0$. Since $\efont^{\gamma=1} = k$ by lemma \ref{pginv}, this implies $f_0\in k$, so that $f \in \phi_q(\efont)$.

However, lemma \ref{ltder} and the fact that $\chi_\pi(\gamma) \equiv 1 \bmod{\pi}$ imply that $\gamma(Y) = Y + f_\gamma(Y^p)$ for some $f_\gamma \in \efont$, so that $\gamma(Y^{q/p}) = Y^{q/p} + \phi_q(g_\gamma)$. Hence $(1-\gamma)(Y^{q/p}) \in \phi_q(\efont)$ even though $Y^{q/p}$ does not belong to $\phi_q(\efont)$. Therefore, no such map $R$ can exist.
\end{proof}

\begin{coro}
\label{noproj}
If $K \neq \Qp$, there is no $\Gamma$-equivariant $k$-linear projector $\phi_q^{-1}(\efont) \to \efont$. 
A fortiori, there is no $\Gamma$-equivariant $k$-linear projector $\et \to \efont$.
\end{coro}

\begin{proof}
Given such a projector $\Pi$, we could define $R$ as in prop \ref{noltrace} by $R = \Pi \circ \phi_q^{-1}$.
\end{proof}

\subsection{The perfectoid commutant of $\Aut(\LT)$}
\label{subpercom}

In \S 3.1 of \cite{BR22}, we computed the perfectoid commutant of $\Aut(\Gm)$. We now use theorem \ref{decprec} to do the same for $\Aut(\LT)$.  We still assume that $p \neq 2$.

\begin{theo}
\label{ltcom}
If $u \in \et^+$ is such that $\vy(u) > 0$ and $u \circ  [g] = [g] \circ u$ for all $g \in \OO_K^\times$, there exists $b \in \OO_K^\times$ and $n \in \ZZ$ such that $u(Y) = [b](Y^{q^n})$.
\end{theo}

Recall that a power series $f(Y) \in k\dcroc{Y}$ is separable if $f'(Y) \neq 0$. If $f(Y) \in Y \cdot k \dcroc{Y}$, we say that $f$ is invertible if $f'(0) \in k^\times$, which is equivalent to $f$ being invertible for composition (denoted by $\circ$). We say that $w(Y) \in Y \cdot k \dcroc{Y}$ is nontorsion if $w^{\circ n}(Y) \neq Y$ for all $n \geq 1$. If $w(Y) = \sum_{i \geq 0} w_i Y^i \in k \dcroc{Y}$ and $m \in \ZZ$, let $w^{(m)}(Y) = \sum_{i \geq 0} w_i^{p^m} Y^i$. Note that $(w \circ v)^{(m)} = w^{(m)} \circ v^{(m)}$.

\begin{prop}
\label{lubnarch}
Let $w(Y) \in Y + Y^2\cdot k \dcroc{Y}$ be a nontorsion series, and let $f(Y) \in Y \cdot k \dcroc{Y}$ be a separable power series. If $w^{(m)} \circ f = f \circ w$ for some $m \in \ZZ$, then $f$ is invertible.
\end{prop}

\begin{proof}
This is a slight generalization of lemma 6.2 of \cite{L94}. Write 
\begin{align*}
f(Y) & = f_n Y^n + \bigO(Y^{n+1}) \\ f'(Y)  &= g_j Y^j + \bigO(Y^{j+1}) \\ w(Y) & = Y + w_r Y^r + \bigO(Y^{r+1}),
\end{align*}
with $f_n,g_j,w_r \neq 0$. Since $w$ is nontorsion, we can replace $w$ by $w^{\circ p^\ell}$ for $\ell \gg 0$ and assume that $r \geq  j+1$. We have 
\begin{align*} w^{(m)} \circ f & = f(Y) + w_r^{(m)} f(Y)^r + \bigO(Y^{n(r+1)})  \\ & = f(Y) + w_r^{(m)} f_n^r Y^{nr} + \bigO(Y^{nr+1}). \end{align*}
If $j=0$, then $n=1$ and we are done, so assume that $j \geq 1$. We have
\begin{align*} f \circ w &= f(Y + w_r Y^r + \bigO(Y^{r+1})) \\ &= f(Y) + w_r Y^r f'(Y)  + \bigO(Y^{2r})  \\ &=  f(Y) + w_r g_j Y^{r+j} + \bigO(Y^{r+j+1}). \end{align*}
This implies that $nr=r+j$, hence $(n-1)r=j$, which is impossible if $r>j$ unless $n=1$. Hence $n=1$ and $f$ is invertible.
\end{proof}

\begin{lemm}
\label{locansh}
If $u \in \et^+$ is such that $\vx(u) > 0$ and $u \circ  [g] = [g] \circ u$ for all $g \in \OO_K^\times$, then $u \in (\et^+)^{d\hyphen\sh}$.
\end{lemm}

\begin{proof}
The group $\OO_K^\times$ acts on $\et^+$ by $g \cdot u = u \circ [g]$. By lemmas \ref{diflt} and \ref{sbokc}, the function $g \mapsto [g] \circ u$ is in $\HH_d^\lambda(1+ p \OO_K, \et^+)$, where $p^\lambda=\vy(u)$.
\end{proof}

\begin{proof}[Proof of theorem \ref{ltcom}]
Take $u \in \et$ such that $\vy(u)>0$ and $u \circ [g] = [g] \circ u$ for all $g \in \OO_K^\times$. By lemma \ref{locansh} and theorem \ref{decprec}, there is an $m \in \ZZ$ such that $f(Y) = u(Y)^{p^m}$ belongs to $Y \cdot k \dcroc{Y}$ and is separable. Take $g \in 1+\pi \OO_K$ such that $g$ is nontorsion, and let $w(Y) = [g](Y)$ so that $u \circ w = w \circ u$. We have $f \circ w = w^{(m)} \circ f$. By prop  \ref{lubnarch}, $f$ is invertible. In addition, $f \circ w = w^{(m)} \circ f$ if $w(Y) = [g](Y)$ for all $g \in \OO_K^\times$. Hence $f_0 \cdot \overline{g} = \overline{g}^{p^m} \cdot f_0$, so that $a^{p^m} = a$ for all $a = \overline{g} \in k$. This implies that $\FF_q \subset \FF_{p^{|m|}}$, so that $m=fn$ for some $n \in \ZZ$. Hence $w^{(m)} =w$, and $f \circ [g] = [g] \circ f$ for all $g \in \OO_K^\times$. Theorem 6 of \cite{LS07} implies that $f \in \Aut(\LT)$. Hence there exists $b \in \OO_K^\times$ such that $u(Y)=[b](Y^{q^n})$.
\end{proof}

\section{Mahler expansions and super-H\"older functions}
\label{mahlsec}

In \S 1.3 of \cite{BR22}, we proved an analogue of Mahler's theorem
for continuous functions $\Zp \to M$, and then gave a characterization of
super-H\"older functions in terms of their Mahler expansions. We now
indicate how these results generalize to functions $G \to M$ for a
uniform pro-$p$ group $G$. 
Given the definition of super-H\"older functions and the existence of a coordinate $c : G \to \Zp^d$ as in prop \ref{satpv}, it is enough to study functions $\Zp^d \to M$. We generalize the setting a little bit, and study functions $\OO_K^d \to M$ where $K$ is a finite extension of $\Qp$. Let $K$ be such a field, fix a uniformizer $\pi$ of $\OO_K$ and let $k$ be the residue field of $K$. Let $q=\Card(k)$. 

\subsection{Good bases and wavelets}
\label{goodbases}

Let $X=\OO_K^d$, which we endow with the valuation $\vx(x_1,\dots,x_d) =\min_i \val_\pi(x_i)$. For  $n \geq 0$, let $X_n = \pi^n X = \{x \in X, \vx(x) \geq n\}$. 

We endow $X$ with the $\vx$-adic topology.
For any set $Y$, we denote by $\LC(X,Y)$ the set of locally constant
functions $X \to Y$. For $n \geq 0$ we denote by $\LC_n(X,Y)$
the subset of elements of $\LC(X,Y)$ that factor through $X/X_n$.
Let $I = \cup_{n \geq 0} I_n$ be a set of indices, where $I_n \subset I_{n+1}$ 
for all $n \geq 0$, and $\Card(I_n) = \Card(X/X_n)=q^{nd}$.
Let $E$ be a field of characteristic $p$.

\begin{defi}
\label{defgood}
A family $\{h_i\}_{i \in I}$ is a good basis of $\LC(X,E)$ if it is a basis
of the $E$-vector space $\LC(X,E)$ such that for all $n \geq 0$, $\{h_i\}_{i\in I_n}$ is
a basis of $\LC_n(X,E)$.
\end{defi}

Let $M$ be (as usual) an $E$-vector
space with a valuation $\vm$, such that $\vm(ax) = \vm(x)$ for all $a \in
E^\times$ and $x \in M$. We assume that $M$ is separated and complete for the
$\vm$-adic topology.

\begin{prop}
\label{basisLC}
Every $f \in \LC_n(X,M)$ can be written uniquely as $\sum_{i \in
I_n}h_i\cdot m_i$ for some elements $m_i \in M$. Moreover, $\inf_{x \in
X}\vm(f(x)) = \inf_{i \in I_n}\vm(m_i)$.
\end{prop}

\begin{proof}
Let $\{\delta_x\}_{x \in X/X_n}$ be the basis of $\LC_n(X,E)$ defined as
follows: $\delta_x$ is the characteristic function of $x+X_n$. 
Then $f \in \LC_n(X,M)$ is equal to $\sum_{x\in X/X_n}\delta_x\cdot
f(x)$.

As $\{h_i\}_{i\in I_n}$ is also a basis of $\LC_n(X,E)$, we can write
$\delta_x = \sum_{i\in I_n}a_{i,x}h_i$ for some elements $a_{i,x} \in E$. 
We now have $f = \sum_{i\in I_n}h_i\cdot m_i$ where $m_i = \sum_{x\in X/X_n}a_{i,x}f(x)$.
This formula implies that $\inf_{i\in I_n}\vm(m_i) \geq \inf_{x\in X}\vm(f(x))$. 

On the other hand we can also write $h_i = \sum_{x\in
X/X_n}b_{x,i}\delta_x$ for some elements $b_{x,i} \in E$, so
that $f(x) = \sum_{i\in I_n}b_{x,i}m_i$. This implies that 
$\inf_{i\in I_n}\vm(m_i) \leq \inf_{x\in X}\vm(f(x))$. 
\end{proof}

We now give an example of a particularly nice good basis of
$\LC(X,E)$, the basis of wavelets (see \S I.3 of \cite{C10} and \S 2.1 of \cite{M16}).
Let $\cT$ be a set of representatives of $X/X_1$ in $X$, chosen so that the
representative of $0$ is $0$. For each $n \geq 0$, let $\cR_n$ be the set of
representatives of $X/X_n$ defined as follows: $\cR_0 = \{0\}$, and for $n \geq
1$, $\cR_n = \{\sum_{i=0}^{n-1}\pi^i x_i$, $x_i \in \cT$ for all $i\}$.
We have $\cR_1 = \cT$, and $\cR_n \subset \cR_{n+1}$ for all $n$.
Let $\cR = \cup_{n \geq 0} \cR_n$. If 
$r \in \cR$ let $\ell(r)$ be the smallest $n$ such that $r \in \cR_n$.
For $r \in \cR$, let $\chi_r$ be the characteristic function of the
closed disc $r+X_{\ell(r)} = \{x \in X, \vx(x-r) \geq \ell(r)\}$. 

\begin{prop}
\label{wavebasis}
The set $\{\chi_r\}_{r \in \cR}$ is a good basis of $\LC(X,E)$.
\end{prop}

\begin{proof}
We prove that for all $n \geq 0$, the set $\{\chi_r\}_{r \in \cR_n}$ is a basis of $\LC_n(X,E)$. 
Consider the basis $\{\delta_r\}_{r\in \cR_n}$ of $\LC_n(X,E)$, where
$\delta_r$ is the characteristic function of $r+X_n$. 
We have 
\[ \chi_r = \sum_{r' \in \cR_{n-\ell(r)}}\delta_{r+\pi^{\ell(r)}r'}. \]

This implies that if we write $\cR_n = (\cR_n \setminus \cR_{n-1}) \sqcup \hdots \sqcup (\cR_1 \setminus \cR_0) \sqcup \cR_0$ and we express the
family $\{\chi_r\}_{r\in \cR_n}$ in terms of the basis $\{\delta_r\}_{r\in \cR_n}$, we
get a unipotent matrix. This shows that $\{\chi_r\}_{r \in \cR_n}$ is
also a basis of $\LC_n(X,E)$.
\end{proof}

\subsection{Expansions of continuous functions}
\label{expcont}

We  show that every continuous function $X \to M$ has a convergent expansion along a good basis of $X$, and prove some continuity estimates in terms of the coefficients of the expansion. If $\{m_i\}_{i\in I}$ is a family of $M$, we say that $m_i \to 0$ if $\inf_{i \notin I_n} \vm(m_i) \to + \infty$ as $n \to + \infty$.

\begin{theo}
\label{expansion}
Let $\{h_i\}_{i \in I}$ be a good basis of $\LC(X,E)$. 

If $\{m_i\}_{i\in I}$ is a family of $M$ such that $m_i \to 0$,  the function $f : X \to M$ given by $f = \sum_{i\in I}h_i \cdot m_i$ belongs to $C^0(X,M)$, and  $\inf_{x\in X} \vm(f(x)) = \inf_{i\in I}\vm(m_i)$.

Conversely, if $f \in C^0(X,M)$, there exists a unique family $\{m_i(f)\}_{i \in I}$ of elements
of $M$ such that $m_i(f) \to 0$ and such that $f = \sum_{i\in I}h_i \cdot m_i(f)$.
\end{theo}

\begin{proof}
Let $\{m_i\}_{i\in I}$ be a family of $M$ such that $m_i \to
0$. If $f_n = \sum_{i \in I_n}h_i\cdot m_i$, then $f_n
\in C^0(X,M)$, and $f$ is the uniform limit of the $f_n$.
We have $\inf_X \vm(f_n(x)) = \inf_{i\in I_n} \vm(m_i)$ by prop
\ref{basisLC}. Since $m_i \to 0$, we have 
$\inf_{i\in I}\vm(m_i) = \inf_{i\in I_n} \vm(m_i)$ for $n \gg 0$.
Hence $\inf_X\vm(f_n(x)) = \inf_{i\in I}\vm(m_i)$ for $n \gg 0$.
Since $\inf_{x\in X}\vm(f(x)) = \lim_n \inf_x\vm(f_n(x))$, we have
$\inf_{x\in X}\vm(f(x)) = \inf_{i\in I}\vm(m_i)$.

We now prove the converse. Let $M_n = \{ m \in M,
\vm(m) \geq n\}$, let $\pi_n : M \to M/M_n$ be the
projection, and for each $n$, fix a lift $\psi_n : M/M_n \to M$. 
Take $f \in C^0(X,M)$, and let $f_n = \psi_n \circ \pi_n \circ f$. As $f$ and $f_n$ co\"incide modulo
$M_n$, $f$ is the uniform limit of the $f_n$. On the other hand,
$\pi_n \circ f$ is locally constant, and therefore so is $f_n$. As $X$ is compact,
there exists some $k(n)
\geq 0$ such that $f_n \in \LC_{k(n)}(X,M)$. By prop \ref{basisLC}, we can write 
$f_n = \sum_{i \in I} h_i \cdot m_{i,n}$, where $m_{i,n} = 0$ if $i \notin I_{k(n)}$.
We have $\vm(m_{i,n}-m_{i,n'}) \geq \min(n,n')$ by construction, 
so that for each $i$, the sequence $\{m_{i,n}\}_n$ converges to some $m_i \in M$.
Moreover, if $i \notin I_{k(n)}$, then $\vm(m_i) \geq n$, so that $m_i \to 0$. 
The continuous function $\sum_{i\in I} h_i \cdot m_i$ is the uniform
limit of the $f_n$, so that finally $f = \sum_{i \in I}h_i \cdot m_i$.
\end{proof}

\begin{prop}
\label{equalgood}
Take $f \in C^0(X,M)$ and $t \in \ZZ_{\geq 0}$. If $\{h_i\}_{i \in I}$ is a good basis of $\LC(X,E)$, and we write $f = \sum_i h_i\cdot m_i$ with $m_i \to 0$, then $\inf_{i\not\in I_t}\vm(m_i)$ depends only on $f$ and not on the choice of the good basis.
\end{prop}

\begin{proof}
Fix two good bases $\{h_i\}_{i \in I}$ and $\{h'_i\}_{i \in I}$ of $\LC(X,E)$. There exists a
family $\{\lambda_{i,j}\}_{(i,j) \in I \times I}$ of elements of $E$ such that $h_i =
\sum_j\lambda_{i,j}h'_j$ for all $i$.
Moreover, if $i \in I_c$
then $\lambda_{i,j} = 0$ for all $j \not\in I_c$.
Now write $f = \sum_{i \in I} h_i \cdot m_i(f) = \sum_{i \in I} h'_i \cdot m_i'(f)$.
We also have \[ f = \sum_i (\sum_j \lambda_{i,j}h'_j) \cdot m_i(f)
= \sum_j h'_j \cdot (\sum_i \lambda_{i,j}m_i(f)), \] 
so that $m'_j(f) = \sum_i \lambda_{i,j}m_i(f)$. If $j \not\in I_t$, then 
$m'_j(f) = \sum_{i \not\in I_t} \lambda_{i,j}m_i(f)$, as $\lambda_{i,j} = 0$
if $i \in I_t$ and $j \not\in I_t$. This implies that 
$\inf_{j \not\in I_t}\vm(m_j'(f)) \geq \inf_{i \not\in
I_t}\vm(m_i(f))$. 

By symmetry, we get that $\inf_{j \not\in I_t}\vm(m_j'(f)) = \inf_{i \not\in
I_t}\vm(m_i(f))$.
\end{proof}

\begin{theo}
\label{equal}
Take $f \in C^0(X,M)$ and $t \in \ZZ_{\geq 0}$. 

If $\{h_i\}_{i \in I}$ is a good basis of $\LC(X,E)$, and we write $f = \sum_i h_i\cdot m_i$ with $m_i \to 0$, then 
\[ \inf_{i\not\in I_t}\vm(m_i)  = \inf_{\substack{x,y \in X \\ \vx(x-y) \geq t}}\vm(f(x)-f(y)). \]
\end{theo}

\begin{proof}
Let $C_t(f) = \inf_{x,y \in X,\vx(x-y) \geq t}\vm(f(x)-f(y))$ and $B_t(f) = \inf_{i \not\in I_t}\vm(m_i)$.

If $x \in X$ and $z \in X_t$, then 
$f(x+z)-f(x) = \sum_{i \in I}\left(h_i(x+z)-h_i(z)\right) \cdot m_i(f)$.
As $h_i \in \LC_t(X,E)$ for $i \in  I_t$, the above equality gives
us \[ f(x+z)-f(x) = \sum_{i \not\in I_t}\left(h_i(x+z)-h_i(z)\right) \cdot m_i(f). \]
This implies that $C_t(f) \geq B_t(f)$.

We now prove the converse inequality. By prop \ref{equalgood},
$B_t(f)$ is independent of the choice of a good basis, and we choose the wavelet basis of
prop \ref{wavebasis}. Write $f = \sum_{r \in \cR}\chi_r \cdot
m_r(f)$, so that we want to show that $\vm(m_r(f)) \geq C_t(f)$ for all $r \notin \cR_t$. 
If $x \in X$, define $g_x : X \to M$
by $g_x(z) = f(x+\pi^tz)-f(x)$, and write $g_x = \sum_{r\in \cR}\chi_r\cdot m_r(g_x)$.  
For each $r \in \cR$, we can write uniquely $r = r_t + \pi^t s$ with $r_t
\in \cR_t$, where $s = 0$ if $r \in \cR_t$, and $s \neq 0 \in
\cR_{\ell(r)-t}$ if $r \notin \cR_t$. For $x \in \cR_t$ and $r \notin \cR_t$, the map
$z \mapsto \chi_r(x+\pi^tz) -\chi_r(x)$ is the zero function if $r_t \neq
x$, and is $\chi_s$ if $r_t = x$. This implies that if $x \in \cR_t$, then
\begin{align*}
g_x(z) &= \sum_{r\in \cR} \left(\chi_r(x+\pi^tz)-\chi_r(x)\right) \cdot
m_r(f) \\
 &= \sum_{r\notin \cR_t} \left(\chi_r(x+\pi^tz)-\chi_r(x)\right) \cdot m_r(f) \\
 &= \sum_{s \notin \cR_0} \chi_s(z) \cdot m_{x +\pi^t s}(f).
\end{align*}
Therefore if $x \in \cR_t$, then $m_0(g_x)=0$ and $m_s(g_x) = m_{x+\pi^t s}(f)$ if $s \neq 0$.
We have $\inf_{s \in \cR} \vm(m_s(g_x)) = \inf_{z \in X}\vm(g_x(z)) \geq C_t(f)$, so that $\vm(m_s(g_x)) \geq C_t(f)$ for all $x \in X$ and $s \in \cR$. This implies that for all $x \in \cR_t$ and $s \neq 0$, $\vm(m_{x+\pi^t s}(f)) \geq C_t(f)$. Hence for all $r \notin \cR_t$, we have $\vm(m_r(f)) \geq C_t(f)$.
\end{proof}

\subsection{Mahler bases}
\label{mahlgen}

We now construct some other examples of good bases. For $n \geq 0$, let $\Int_n(\OO_K)$ denote the set of polynomials $f(T)
\in K[T]$ such that $\deg(P) \leq n$ and $f(\OO_K) \subset \OO_K$. Recall
(see for instance \S 1.2 of \cite{M16}) that a Mahler basis for $\OO_K$ is a sequence $\{h_n\}_{n
\geq 0}$ with $h_n(T) \in K[T]$ of degree $n$, and such that
$\{h_0,\hdots,h_n\}$ is a basis of the free $\OO_K$-module
$\Int_n(\OO_K)$ for all $n \geq 0$. For example, if $K=\Qp$, we can take
$h_n(T) = \binom{T}{n}$. Let $\{h_n\}_{n \geq 0}$ be a Mahler basis for $\OO_K$. Each
$h_n$ defines a function $\OO_K \to \OO_K$ and hence $\OO_K \to k$.
Let $I=\ZZ_{\geq 0}$ and let $I_n = \{ 0, \hdots, q^n-1\}$ for $n \geq 0$.

\begin{prop}
\label{dsmb}
If $\{h_n\}_{n \geq 0}$ is a Mahler basis for $\OO_K$, then $\{h_i\}_{i \in I}$ is a good basis of $\LC(\OO_K,k)$.
\end{prop}

\begin{proof}
By theorem 1.2 of \cite{M16}, $\{h_0,\hdots,h_{q^m-1}\}$ is a basis of 
the $k$-vector space $\LC_m(\OO_K,k)$ for all $m \geq 0$. This
implies the claim.
\end{proof}

We now specialize to $K=\Qp$. 
Write $\NN$ for $\ZZ_{\geq 0}$ and $\nbf$ for an element $(n_1,\hdots,n_d) \in \NN^d$.
For each $\nbf \in \NN^d$, we denote by $h_\nbf$ the
function $\Zp^d \to E$ given by $(x_1,\dots,x_d) \mapsto \binom{x_1}{n_1}\cdots
\binom{x_d}{n_d}$. 
For $m \in \ZZ_{\geq 0}$, let $I_m = \{\nbf \in \NN^d$ such that $\max(n_1,\hdots,n_d) \leq p^m-1\}$.

\begin{prop}
\label{goodbasisZpd}
The functions $\{h_\nbf\}_{\nbf\in \NN^d}$ form a
good basis of $\LC(\Zp^d,\FF_p)$.
\end{prop}

\begin{proof}
The claim follows from
prop \ref{dsmb} for $K = \Qp$, and lemma \ref{prodgood} below.
\end{proof}

\begin{lemm}
\label{prodgood}
If $X$ and $X'$ are as in \S \ref{goodbases}, and $\{h_i\}_{i \in I}$ and $\{h'_j\}_{j \in J}$ are 
good bases of $\LC(X,E)$ and $\LC(X',E)$, then $\{ h_i \otimes h'_j\}_{(i,j) \in I \times J}$ is a good basis 
of $\LC(X \times X',E)$, with $(I \times J)_n = I_n \times J_n$.
\end{lemm}

Let $G$ be a uniform pro-$p$ group, and let $c : G \to \Zp^d$ be a coordinate as in prop \ref{satpv}. The theorem below follows from prop \ref{goodbasisZpd}, theorem \ref{expansion}, and theorem \ref{equal}.

\begin{theo}
\label{mahlthm}
If $\{m_\nbf\}_{\nbf \in \NN^d}$ is a sequence of $M$ such that $m_\nbf
\to 0$, the function $f : G \to M$ given by $f(g)
= \sum_{\nbf \in \NN^d} \binom{c_1(g)}{n_1} \cdots \binom{c_d(g)}{n_d}
m_\nbf$ belongs to $C^0(G,M)$. We have $\inf_{g \in G} \vm(f(g))
=\inf_{\nbf \in \NN^d} \vm(m_\nbf)$.

Conversely, if $f \in C^0(G,M)$, there exists a unique sequence
$\{m_\nbf(f)\}_{\nbf \in \NN^d}$ such that $m_\nbf(f) \to 0$ 
and such that $f(g) = \sum_{\nbf \in \NN^d}
\binom{c_1(g)}{n_1} \cdots \binom{c_d(g)}{n_d} m_\nbf(f)$.  

We have $f \in \HH_e^{\lambda,\mu}(G,M)$ if and only if
for all $i \geq 0$, we have $\vm(m_\nbf(f)) \geq p^\lambda \cdot p^{ei}
+\mu$ whenever $\max(n_1,\hdots,n_d) \geq p^i$.
\end{theo}

\begin{rema}
\label{ardakomment}
The first two assertions in the above theorem
also follow from theorem 1.2.4 in \S III of \cite{L65} (we thank Konstantin Ardakov for pointing this out).
\end{rema}

We finish by considering the case $G=\OO_K$ for $K$ a finite extension of $\Qp$, 
and working with a Mahler basis for $\OO_K$. Let $K$ be a finite extension of $\Qp$ as before.
Assume that $E$ is an extension of $k$. Let $\{h_n\}_{n \geq 0}$ be a Mahler basis for $\OO_K$. If $f \in C^0(\OO_K,M)$, write $f = \sum_{n \geq 0} h_n m_n(f)$ with $m_n(f) \to 0$. Let $e$ denote the ramification index of $K$.

\begin{prop}
\label{okansh}
If $f = \sum_{n \geq 0} h_n m_n(f)$ as above, then
$f \in \HH_t^{\lambda,\mu}(\OO_K,M)$ if and only if
$\vm(m_n(f)) \geq p^\lambda \cdot p^{ti}+\mu$ whenever
$n \geq p^{di}$.
\end{prop}

\begin{proof}
This follows from theorem \ref{equal}, since $\vp(x-y) \geq i$ if and
only if $\val_\pi(x-y) \geq ei$, and since $q^e = p^d$.
\end{proof}

In this situation we can also define a slightly different version of
super-H\"older functions. We say that a function $f : \OO_K \to M$ is 
in $\HH_{K,t}^{\lambda,\mu}(\OO_K,M)$ if 
$\vm(f(x) - f(y)) \geq p^\lambda\cdot p^{ti}+\mu$ whenever $\val_\pi(x-y)
\geq i$. We then have
\[ \HH_{te}^{\lambda+t(e-1),\mu}(\OO_K,M) \subset
\HH_{K,t}^{\lambda,\mu}(\OO_K,M) \subset
\HH_{te}^{\lambda,\mu}(\OO_K,M). \]
In particular, $\HH_{K,t}(\OO_K,M) = \HH_{te}(\OO_K,M)$.
If $K/\Qp$ is unramified then
$\HH_{K,t}^{\lambda,\mu}(\OO_K,M) =
\HH_{t}^{\lambda,\mu}(\OO_K,M)$.
Moreover we have the following criterion:

\begin{prop}
If $f = \sum_{n \geq 0} h_n m_n(f)$ as above,  then $f \in \HH_{K,t}^{\lambda,\mu}(\OO_K,M)$ if and only if
$\vm(m_n(f)) \geq p^\lambda \cdot p^{ti}+\mu$ whenever
$n \geq q^i$.
\end{prop}

\begin{exem}
\label{cnlt}
For all $n \geq 0$, there exists $c_n(T) \in \Int_n(\OO_K)$ such that $[a](Y) = \sum_{n \geq 0} c_n(a) Y^n$. 
This implies that $\vy\left(m_n(a \mapsto [a](Y))\right) \geq n$, so that the function $a \mapsto [a](Y)$ is in
$\HH_d^{0,0}(\OO_K,E\dcroc{Y})$, and in $\HH_{K,f}^{0,0}(\OO_K,E\dcroc{Y})$ where
$q = p^f$.
\end{exem}

%

\providecommand{\bysame}{\leavevmode ---\ }
\providecommand{\og}{``}
\providecommand{\fg}{''}
\providecommand{\smfandname}{\&}
\providecommand{\smfedsname}{\'eds.}
\providecommand{\smfedname}{\'ed.}
\providecommand{\smfmastersthesisname}{M\'emoire}
\providecommand{\smfphdthesisname}{Th\`ese}

\end{document}